\documentclass[11pt]{article}

\usepackage{amsmath}
\usepackage{amsfonts}
\usepackage{amssymb}
\usepackage{amsthm}
\usepackage{breqn}
\usepackage{setspace}
\usepackage{fullpage}
\usepackage{enumitem}
\usepackage{bbold} 
\usepackage{comment}
\usepackage{hyperref}

\newtheorem{lemma}{Lemma}
\newtheorem{theorem}{Theorem} 
 
\newtheorem{definition}{Definition}
\newtheorem{claim}{Claim}

\newtheorem{note}{Note}

\newcommand{\A}{\mathcal{A}}
\newcommand{\B}{\mathcal{B}}
\newcommand{\C}{\mathcal{C}}
\newcommand{\D}{\mathcal{D}}
\newcommand{\U}{\mathcal{U}}
\newcommand{\I}{\mathcal{I}}
\newcommand{\T}{\mathcal{T}}
\newcommand{\M}{\mathcal{M}}
\newcommand{\F}{\mathcal{F}}
\newcommand{\G}{\mathcal{G}}
\newcommand{\h}{\mathcal{H}}
\newcommand{\LL}{\mathcal{L}}

\newcommand{\abs}[1]{\left\lvert{#1}\right\rvert}
\newcommand{\floor}[1]{\left\lfloor{#1}\right\rfloor}

\DeclareMathOperator{\La}{La}
\DeclareMathOperator{\ave}{ave}

\title{Intersecting P-free families}
\author
{
D\'aniel Gerbner \thanks{Alfr\'ed R\'enyi Institute of Mathematics, Hungarian Academy of Sciences  e-mail: gerbnerd@gmail.com}
\and
Abhishek Methuku \thanks{Central European University, Budapest, Hungary e-mail: abhishekmethuku@gmail.com}
\and 
Casey Tompkins \thanks{Central European University, Budapest, Hungary} \thanks{Alfr\'ed R\'enyi Institute of Mathematics, Hungarian Academy of Sciences e-mail: ctompkins496@gmail.com}
}

\begin{document}
\maketitle

\begin{abstract}
We study the problem of determining the size of the largest intersecting $P$-free family for a given partially ordered set (poset) $P$.  In particular, we find the exact size of the largest intersecting $B$-free family where $B$ is the butterfly poset and classify the cases of equality. The proof uses a new generalization of the partition method of Griggs, Li and Lu. We also prove generalizations of two well-known inequalities of Bollob\'{a}s and Greene, Katona and Kleitman in this case. Furthermore, we obtain a general bound on the size of the largest intersecting $P$-free family, which is sharp for an infinite class of posets originally considered by Burcsi and Nagy, when $n$ is odd. Finally, we give a new proof of the bound on the maximum size of an intersecting $k$-Sperner family and determine the cases of equality.
\end{abstract}

\section{Introduction}
We denote the set $\{1,2,\dots,n\}$ by $[n]$ and the power set of $[n]$ by $2^{[n]}$.  The family of all $k$-element subsets of $[n]$ is denoted by $\binom{[n]}{k}$.  We refer to  $\binom{[n]}{k}$ as the $k^{th}$ \emph{level} in $2^{[n]}$.  A collection $\F \subseteq 2^{[n]}$ is called an \emph{antichain} if there do not exist $F,G\in\F$ with $F \subset G$.   Let $P$ and $Q$ be partially ordered sets (posets).  Then, $P$ is said to be a \emph{subposet} of $Q$ if there exists an injection $\phi$ from $P$ to $Q$ such that $x < y$ in $P$ implies $\phi(x) < \phi(y)$ in $Q$.  Note, importantly, that the implication is only required in one direction.  

The starting point for all forbidden poset problems is the well-known theorem of Sperner~\cite{sperner1928satz}:

\begin{theorem}[Sperner \cite{sperner1928satz}]
Let $\F \subseteq 2^{[n]}$ be an antichain, then
\begin{displaymath}
\abs{\F} \le \binom{n}{\floor{\frac{n}{2}}}.
\end{displaymath}
Moreover, equality occurs if and only if $\F$ is a level of maximum size in $2^{[n]}$.
\end{theorem}

Observe that every collection $\F \subseteq 2^{[n]}$ may itself be viewed as a poset under the containment relation.  A $k$-chain, denoted by $P_k$, is defined to be the poset on the set $\{x_1,x_2,\dots,x_k\}$ with the relations $x_1< x_2<\dots < x_k$.  Sperner's theorem is equivalent to the statement that the size of a collection $\F \subseteq 2^{[n]}$ containing no 2-chain as a subposet is at most $\binom{n}{\floor{\frac{n}{2}}}$.

An important generalization of Sperner's theorem, due to Erd\H{o}s~\cite{erdos1945lemma}, determines the size of the largest family containing no $(k+1)$-chain.  Such a family is called $k$-\emph{Sperner}.  We use the notation $\Sigma(n,k)$ to denote the sum of the $k$ largest binomial coefficients of the form $\binom{n}{i}$,  $0\le i \le n$.

\begin{theorem}[Erd\H{o}s \cite{erdos1945lemma}]
Let $\F \subseteq 2^{[n]}$ be  $k$-Sperner, then
\begin{displaymath}
\abs{\F} \le \Sigma(n,k).
\end{displaymath}
Moreover, equality occurs if and only if $\F$ is the union of $k$ of the largest levels in $2^{[n]}$.
\end{theorem}

The general study of forbidden poset problems was initiated in the paper of Katona and Tarj\'an~\cite{katona1983extremal}.   They determined the size of the largest family of sets containing neither a $V$ (the poset on $\{x,y,z\}$ with relations $x < y,z$) nor a $\Lambda$ (the poset on $\{x,y,z\}$ with relations $x,y < z$).  They also gave an estimate on the maximum size of $V$-free families which we will make use of.

\begin{theorem}[Katona, Tarj\'an \cite{katona1983extremal}]   
\label{thm:KT}
Assume that $\F \subseteq 2^{[n]}$ contains no $V$ as a subposet, then
\begin{displaymath}
\abs{\F} \le \left(1 + \frac{2}{n}\right)\binom{n}{\floor{\frac{n}{2}}}.
\end{displaymath}
\end{theorem} 
The following function is the main object of study in forbidden poset problems:

\begin{displaymath}
\La(n,P) = \max_{\F \subseteq 2^{[n]}} \{\abs{\F}: \F \mbox{ does not contain } P \mbox{ as a subposet}\}.
\end{displaymath}

The value of $\La(n,P)$ has been determined or estimated for a variety of posets $P$. The butterfly poset, $B$, is defined on the set $\{w,x,y,z\}$ with relations $w,x < y,z$.  Of central importance to the present paper is a result of De Bonis, Katona and Swanepoel~\cite{de2005largest} which gave the exact result for $\La(n,B)$.

\begin{theorem}[De Bonis, Katona, Swanepoel \cite{de2005largest}]
\label{thm:DKS}
\begin{displaymath}
\La(n,B) = \Sigma(n,2).
\end{displaymath}
Moreover, equality holds if and only if the family is the union of two of the largest levels in $2^{[n]}$.  
\end{theorem}
 
A family $\F$ is called \emph{intersecting} if for any two members $F,F'\in \F$ we have $F \cap F'\neq \emptyset$, and it is called $t$-\emph{intersecting} if for any two members $F,F'\in \F$ we have $|F\cap F'|\ge t$.
Now we will mention some theorems where there is a forbidden subposet and the family is also required to be intersecting.  Milner~\cite{milner1968combinatorial} determined the size of the largest $t$-intersecting antichain.  In the case $t=1$, Milner's result yields

\begin{theorem}[Milner \cite{milner1968combinatorial}]
Let $\F \subseteq 2^{[n]}$ be an intersecting antichain,  then
\begin{displaymath}
\abs{\F} \le \binom{n}{\floor{\frac{n}{2}}+1}.
\end{displaymath}
\end{theorem}
This result follows from a more general inequality of Greene, Katona and Kleitman~\cite{greene1976extensions} (See also \cite{katona1998simple} and \cite{scott1999another} for other simple proofs).  

\begin{theorem}[Greene, Katona, Kleitman \cite{greene1976extensions}]
\label{gkk}
Let $\F \subseteq 2^{[n]}$ be an intersecting antichain, then
\begin{displaymath}
\sum_{\substack{F\in \F \\ \abs{F} \le \frac{n}{2}}} \frac{1}{\binom{n}{\abs{F}-1}}+ \sum_{\substack{F\in \F \\ \abs{F} > \frac{n}{2}}} \frac{1}{\binom{n}{\abs{F}}} \le 1.
\end{displaymath}
\end{theorem}

In the case when $\F$ consists of only sets of size at most $\frac{n}{2}$, Bollob\'{a}s~\cite{bollobas1973sperner} proved a stronger inequality generalizing the Erd\H{o}s-Ko-Rado theorem~\cite{erdos1961intersection}.

\begin{theorem}[Bollob\'{a}s \cite{bollobas1973sperner}]
\label{bollobas}
Let $\F \subseteq 2^{[n]}$ be an intersecting antichain and assume that for all $F \in \F$ we have $\abs{F} \le \frac{n}{2}$, then
\begin{displaymath}
\sum_{F \in \F} \frac{1}{\binom{n-1}{\abs{F}-1}} \le 1.
\end{displaymath}
\end{theorem}
Note that Theorems \ref{gkk} and \ref{bollobas} are implied by a more general result of P\'{e}ter Erd\H{o}s, Frankl and Katona \cite{erdHos1984intersecting} which determined the profile polytope for intersecting antichains.

In the course of determining the profile polytope for complement-free $k$-Sperner families, Gerbner~\cite{gerbner2013profile} proved a generalization of Milner's theorem (the 1-intersecting case) to the $k$-Sperner setting.

\begin{theorem}[Gerbner \cite{gerbner2013profile}]
\label{kchainint}
Let $\F \subseteq 2^{[n]}$ be an intersecting $k$-Sperner family, then
\begin{equation}
\label{dani}
\abs{\F} \le \begin{cases}
\sum\limits_{i = \frac{n+1}{2}}^{\frac{n+1}{2} + k -1} \binom{n}{i}, &\text{if $n$ is odd} \\
\binom{n-1}{\frac{n}{2}-1} + \sum\limits_{i=\frac{n}{2}+1}^{\frac{n}{2}+k-1} \binom{n}{i} + \binom{n-1}{\frac{n}{2}+k},  &\text{if $n$ is even.}
\end{cases}
\end{equation}
\end{theorem}

For simplicity, we denote the right-hand side of \eqref{dani} by $\sum_I(n,k)$.  For any given $P$, we define
\begin{displaymath}
\La_I(n,P) = \max_{\F \subseteq 2^{[n]}} \{\abs{\F}: \F \mbox{ does not contain } P \mbox{ as a subposet and } \F \mbox{ is intersecting}\}.
\end{displaymath}

In this language, Theorem~\ref{kchainint}  states that $\La_I(n,P_{k+1}) = \sum_I(n,k)$, where $P_{k+1}$ is the path poset of length $k+1$.  Before we state our main results we need to introduce some notation. For all $n$ and $k \le \frac{n+1}{2}$, define 
\begin{displaymath}
	\h_{0,n,k} = \binom{[n]}{\floor{\frac{n}{2}}+1} \cup \binom{[n]}{\floor{\frac{n}{2}}+2} \cup \dots \cup  \binom{[n]}{\floor{\frac{n}{2}}+k}, 
\end{displaymath}
and in the case when $n$ is even, for any $x \in [n]$, define
\begin{displaymath}
	\h_{x,n,k} = \{F \in \binom{[n]}{\frac{n}{2}}:x \in F\} \cup \binom{[n]}{\frac{n}{2}+1} \cup \dots \cup \binom{[n]}{\frac{n}{2}+k-1} \cup \{F \in \binom{[n]}{\frac{n}{2}+k}:x \not\in F\}.
\end{displaymath}

We determine the exact value of $\La_I(n,B)$, the maximum size of an intersecting butterfly-free family, for $n \ge 17$.   In particular, we show that $\La_I(n,B) = \Sigma_I(n,2)$. The cases of equality are also obtained.

\begin{theorem}
	\label{intBfree}
	Let $\F \subseteq 2^{[n]}$ be an intersecting $B$-free family of subsets of $[n]$ where $n \ge 17$.   Then,
	\begin{displaymath}
	\abs{\F} \le \Sigma_I(n,2).
	\end{displaymath}
	Equality holds if and only if:
	\begin{itemize}
		\item For $n$ odd, $\F = \h_{0,n,2}$;
		\item For $n$ even, $\F = \h_{x,n,2}$ for some $x \in [n]$.
	\end{itemize}
\end{theorem}
The proof of this theorem can be seen as a generalization of the partition method of Griggs, Li and Lu \cite{griggs2012diamond, griggs2013partition} to a weighted setting involving cyclic permutations. We also show that a variant of the LYM-type inequalities, Theorems \ref{gkk} and \ref{bollobas}, hold in this case. 

\begin{theorem}
	\label{gkk2}
	Let $\F \subseteq 2^{[n]}$ be an intersecting $B$-free family with $2 \le \abs{F} \le n-2$ for all $F \in \F$, then
	\begin{displaymath}
	\sum_{\substack{F\in \F \\ \abs{F} \le \frac{n}{2}}} \frac{1}{\binom{n}{\abs{F}-1}}+ \sum_{\substack{F\in \F \\ \abs{F} > \frac{n}{2}}} \frac{1}{\binom{n}{\abs{F}}} \le 2.
	\end{displaymath}
\end{theorem}

\begin{theorem}
	\label{bollobas2}
	Let $\F \subseteq 2^{[n]}$ be an intersecting $B$-free family with $2 \le \abs{F} \le \frac{n}{2}$ for $F \in \F$, then
	\begin{displaymath}
	\sum_{F \in \F} \frac{1}{\binom{n-1}{\abs{F}-1}} \le 2.
	\end{displaymath}
\end{theorem}

Next we obtain an upper bound on $\La_I(n,P)$ for an arbitrary poset $P$ in the case when $n$ is odd. Let $h(P)$ be the height of the poset $P$, that is, the size of the longest chain in $P$. 

\begin{theorem}
	\label{oddcase}
	Assume $n$ is odd and $\frac{\abs{P} + h(P)}{2}$ is an integer. Let $\F$ be an intersecting $P$-free family of subsets of $[n]$, $n \ge 4$. Then, 
	\begin{displaymath}
	\abs \F \le \sum_{i = 1}^{\frac{\abs{P} + h(P)}{2}-1}  \binom{n}{\floor{\frac{n}{2}}+i}.
	\end{displaymath}
\end{theorem}

\begin{note}
	Let $e(P)$ denote the maximum number of consecutive levels in $2^{[n]}$ which do not contain a copy of $P$ as a subposet for any $n$.   Burcsi and Nagy determined the exact value of $\La(n,P)$ for  infinitely many posets $P$ for which $e(P) = \frac{\abs{P} + h(P)}{2} -1$.  For all these posets we have equality in Theorem \ref{oddcase}. In the cases where $n$ is even or $\frac{\abs{P} + h(P)}{2}$ is not an integer, a similar bound can be obtained, but it is not sharp in general. 
\end{note}

Finally, we give a new proof of Theorem \ref{kchainint} which avoids the usage of profile polytopes. We also classify the cases of equality. 

\begin{theorem}
	\label{kspint}
	Let $\F$ be an intersecting $k$-Sperner family of subsets of $[n]$.  Then,
	\begin{displaymath}
	\abs{\F} \le \Sigma_I(n,k).
	\end{displaymath}
	If $k< \frac{n}{2}$, then equality holds in the following cases:
	\begin{itemize}
		\item For $n$ odd, $\F = \h_{0,n,k}$;
		\item For $n$ even and $k=1$, $\F = \h_{0,n,1}$ or $\h_{x,n,1}$ for some $x \in [n]$;
		\item For $n$ even and $k>1$, $\F = \h_{x,n,k}$ for some $x \in [n]$.
	\end{itemize}
	If $k= \frac{n+1}{2}$, then equality holds if and only if $\F=\h_{0,n,k}$.
\end{theorem}
\begin{note}
In the remaining cases there are several extremal families, and we will not characterize them all. For illustration, we mention a few of them.  If $k>\frac{n}{2}$, the upper bound is $2^{n-1}$. If $n$ is odd and $k>\frac{n+1}{2}$, we may take any intersecting family on level $\frac{n-1}{2}$, take every set on level $\frac{n+1}{2}$ that is not a complement of a set taken earlier, and all complete levels from $\frac{n+3}{2}$ to $n$. If $n$ is even and $k>\frac{n}{2}$ we may take any maximal intersecting family on level $\frac{n}{2}$ (of which there are many) in addition to all complete levels from $\frac{n}{2}+1$ to $n$. Finally if $k=\frac{n}{2}$, the upper bound is $2^{n-1}-1$. We may remove $[n]$ from any of the above families to get an extremal intersecting $\frac{n}{2}$-Sperner family.

\end{note}

The paper is organized as follows. In Section \ref{Sectioncycle} we introduce Katona's cycle method \cite{katona1972simple}  and prove some simple lemmas. In Section \ref{SectionintBfree} we prove Theorem \ref{intBfree} determining the exact value of $\La_I(n,B)$. In Section \ref{ButterflyBollobas} we prove Theorem \ref{gkk2} and Theorem \ref{bollobas2}. In Section \ref{generalP} we prove Theorem~\ref{oddcase} about general posets $P$. Finally, in Section \ref{sectionksp} we prove Theorem \ref{kspint} about intersecting $k$-Sperner families.

\section{Cycle method}
\label{Sectioncycle}

A \emph{cyclic permutation} of $[n]$ (in the sense of Katona~\cite{katona1972simple}) is an arrangement of the numbers $1$ through $n$ along a circle.  Sets of consecutive elements along the circle are called \emph{intervals}. The collection of all intervals along $\sigma$
of size $r$ is denoted $\mathcal{L}^{\sigma}_r$. Most of our proofs will proceed by double counting pairs $(F,\sigma)$ where $F \in \F$ and $\sigma$ is a cyclic permutation. Moreover, we will always assume $\varnothing, [n] \not\in \F$  (we handle the remaining cases separately).  For any collection $\h$ of sets, let $\h^\sigma = \{F:F\in \h \mbox{ and $F$ is an interval along $\sigma$}\}$.   In the double counting we will use the following weight function:

\begin{displaymath}
	w(F,\sigma) = 
	\begin{cases}
	\binom{n}{\abs{F}}, &\mbox{ if $F\in \F$ and $F$ is an interval along $\sigma$} \\
	0, &\mbox{ otherwise}.
	\end{cases}
\end{displaymath}
Observe that, on the one hand, we have
\begin{displaymath}
	\sum_{F \in \F} \sum_\sigma w(F,\sigma) =\sum_{F \in \F} \abs{F}!(n-\abs{F})! \binom{n}{\abs{F}} = n! \abs{\F}.
\end{displaymath}
On the other hand,
\begin{displaymath}
	\sum_\sigma \sum_{F \in \F} w(F,\sigma) = \sum_\sigma \sum_{F \in \F^\sigma} \binom{n}{\abs{F}}.
\end{displaymath}

For notational simplicity we will often work with the simplest case of a cyclic permutation where the numbers $1,2,\dots,n$ occur in that order. We call this cyclic permutation the \emph{canonical cyclic permutation}.  It is clear that when we are working with one fixed cyclic permutation we may assume it is canonical because renaming the elements will not change the intersection or containment structure of its intervals. Let $A_i^j$ denote the interval $\{i,i+1,\dots,i+j-1\}$ (addition involving the base set is always taken modulo $n$ except when the result is $0 \bmod n$ which we take to be $n$) where $i$ is called the first element of $A_i^j$ and $i+j-1$ is called the last element of $A_i^j$. We can partition all intervals along $\sigma$ into chains $\C_1,\dots,\C_n$ where $\C_i = \{\{i\},\{i,i+1\},\dots,\{i,i+1,\dots,i+n-1\}\}$. We call this partition the \emph{canonical chain decomposition}. It will be helpful in proving the following well-known result.

\begin{lemma}
	\label{cycleantichain}
	Let $\G$ be an antichain of intervals along a cyclic permutation $\sigma$, then $\abs{\G} \le n$, and equality holds if and only if $\G=\mathcal{L}^{\sigma}_r$ for some $r$.
\end{lemma}

\begin{proof}
We may assume that $\sigma$ is canonical. Let us consider the canonical chain decomposition.  Since at most one interval from each chain may be in our collection, we have that either we take fewer than $n$ intervals or every chain contains exactly one interval from $\G$.  Suppose we are in the latter case and that some two intervals in $\G$ had different sizes.  Then, there must exist chains $\C_i$ and $\C_{i+1}$ where the interval we take in $\C_i$ is larger than the one we take in $\C_{i+1}$.  That is, we have $A_i^{j_1}, A_{i+1}^{j_2} \in \G$ with $j_1>j_2$.  However, this implies that we have $A_{i+1}^{j_2} \subseteq A_i^{j_1}$, a contradiction.
\end{proof}

If we add the additional constraint that the intervals are intersecting and assume that they are of size at most $\frac{n}{2}$, then we have the following better bound (following Katona \cite{katona1972simple}).

\begin{lemma}
	\label{cycleantichain_atmostnhalf}
	Let $\G$ be an intersecting antichain of intervals along a cyclic permutation $\sigma$ where all the intervals are of size at most $\frac{n}{2}$, then $\abs{\G} \le \frac{n}{2}$.
\end{lemma}

\begin{proof}
Suppose without loss of generality that the interval $A_1^k = \{1,2,\dots,k\}$ is in  $\G$. Since $\G$ is intersecting, every interval of $\G$ has either its first element or its last element in $A_1^k$. Also notice that if $i \in \{1,2,\dots,k\}$ is the last element of an interval of $\G$, $i+1$ cannot be the first element of another interval of $\G$ since all the intervals are of size at most $\frac{n}{2}$. Therefore, the total number of intervals in $\G$ is at most $1 + (k - 1) = k \le \frac{n}{2}$, as desired.
\end{proof}

Let $\sigma$ be canonical, and $\G$ be a collection of intervals along $\sigma$. If $\G$ contains only intervals of size $j$ of the form $A_i^j,A_{i+1}^j,\dots, A_{i+s}^j$ for some $0 \le s \le n-1$, then we say that $\G$ is \emph{contiguous}.  If $\G$ is a collection consisting of intervals $A_i^j,A_{i+1}^j,\dots,A_{i+s}^j,A_{i+s+1}^{j+1},A_{i+s+2}^{j+1},\dots,A_{i-2}^{j+1}$ for some $0 \le s \le n-3$, then we say $\G$ is \emph{pair-contiguous}. We will refer to the intervals  $A_i^j,A_{i+1}^j,\dots,A_{i+s}^j$ as the \emph{lower intervals} in $\G$ and the intervals $A_{i+s+2}^{j+1},A_{i+s+3}^{j+1}\dots,A_{i-2}^{j+1}$ as the \emph{upper intervals} in $\G$.  Equivalently, $\G$ is pair-contiguous if it is an antichain, has size $n-1$, and is the union of two contiguous collections of intervals spanning two consecutive levels.    We extend these definitions to arbitrary cyclic permutations in the obvious way.

\begin{lemma}
	\label{2levels}
	If $\G$ is an antichain of intervals along a cyclic permutation $\sigma$ such that $\abs{\G} = n-1$ and $\G$ contains intervals of at least two sizes,  then $\G$ is pair-contiguous.
\end{lemma}
\begin{proof}
	Assume that $\sigma$ is canonical.  Let us consider the canonical chain decomposition. Let $\G_{min}$ be the collection of those intervals in $\G$ of minimum size, say $j^{*}$. Since $\G_{min}$ is not a full level, there must be an $i$ such that $A_i^{j^*} \in  \G_{min}$ but $A_{i-1}^{j^*} \not\in \G_{min}$. Then, we know that $\C_{i-1}$ has no interval from $\G$, and if $\abs{\G}=n-1$ it must be that each chain $\C_i,\C_{i+1},\dots,\C_{i-2}$ contains an interval from $\G$.  Observe that if $\G$ contains an interval of size $j_1$ in $\C_{i_1}$ and an interval of size $j_2$ in $\C_{i_1+1}$, then $j_1 \le j_2$ for otherwise we would not have an antichain.  Finally, the interval from $\G$ in $\C_{i-2}$ must have size $j^*+1$ for if it were any larger it would contain $A_i^{j^*}$.   It follows that $\G$ is a pair-contiguous family contained in levels $j^*$ and $j^*+1$.
\end{proof}

We call a member of $\G$ isolated, if it is comparable with no other member of $\G$.

\begin{lemma}\label{isol}
Let $\G$ be a $2$-Sperner family of intervals on a cyclic permutation with $I$ isolated intervals, then there are at most $2n-I$ intervals in $\G$.
\end{lemma}

\begin{proof}
Consider the canonical chain decomposition. Each isolated interval is found on a different one of the chains.   The remaining chains can have at most 2 intervals each.  It follows that the total number of intervals is at most $I + 2(n-I) = 2n-I$.
\end{proof}

\begin{lemma}
\label{isol_improvement}
Let $\G$ be a $2$-Sperner family of intervals on a cyclic permutation with $I$ isolated intervals, where $1 \le I \le n-1$.  Then, there are at most $2n-I-1$ intervals in total.  
\end{lemma}
\begin{proof}
Consider the canonical chain decomposition.  Assume first that there are $\ell \ge 1$ chains containing either no interval from our family or one interval which is not isolated. Then, the total number of intervals is at most $I + \ell + 2(n-I-\ell) \le I + 1 + 2(n-I-1) = 2n - I -1$.  Indeed, $I+\ell$ chains have at most one interval, and the remaining chains have at most $2$.  Thus, we may assume that every chain contains either two intervals or an isolated interval.   Let $\A$ be the set of inclusion minimal intervals.  Since every chain has two intervals or an isolated interval, $\A$ consists of isolated intervals and the smaller intervals from chains with two intervals (as these are obviously minimal).  Thus, $\abs{\A} = n$, and it follows that $\A$ consists of intervals of only one size.  This yields a contradiction since we know we will have a chain with two intervals from $\A$ followed by an isolated interval (since $1 \le I \le n-1$), but this isolated interval will be contained in the larger interval from the previous chain.
\end{proof}

\section{Intersecting $B$-free families}
\label{SectionintBfree}

In this section we prove Theorem \ref{intBfree} by determining the exact value of $\La_I(n,B)$ and classifying the extremal families.   We may assume that $[n] \not\in \F$.  Indeed, if $[n] \in \F$, then $\F \setminus \{[n]\}$ contains no three sets $A,B,C$ with $A,B \subset C$.  In this case, Theorem \ref{thm:KT} applied to the family of complements shows that such a family may have size at most $(1 + \frac{2}{n}) \binom{n}{\floor{\frac{n}{2}}}$.  Thus, for $n \ge 7$ the family will be too small.

Let  $\F_m = \{F \in \F : \exists A, B \in \F \text{ such that } A \subset F \subset B\}$ (notice that $A$ and $B$ are unique since $\F$ is butterfly-free).  We refer to $\F_m$ as the collection of \emph{middle} sets in $\F$.  Fix a cyclic permutation~$\sigma$.  We will distinguish four kinds of intervals in $\F^\sigma$ which we refer to as the \emph{middle, isolated, top} and \emph{bottom} intervals along $\sigma$.
\begin{align*}
\M_\sigma &= \{F:F \in \F^\sigma \mbox{ and there exists } A,B \in \F^\sigma \mbox{ such that } A \subset F \subset B\}; \\
\I_\sigma &= \{F:F \in \F^\sigma \mbox{ and $F$ is comparable with no other interval in $\F^\sigma$}\}; \\
\T_\sigma &= \{F:F \in \F^\sigma \setminus \I_\sigma \mbox{ is inclusion maximal in } \F^\sigma\};\\
\B_\sigma &=\{F:F \in \F^\sigma \setminus  \I_\sigma \mbox{ is inclusion minimal in } \F^\sigma\}.
\end{align*}
It is easy to see that these four collections of intervals form a partition of $\F^\sigma$.  Importantly, note that the four collections are defined by their properties as intervals along $\sigma$, not in $\F$ itself.  So we may have, for example, a set $F \in \F_m$ which is an interval along $\sigma$, but does not belong to $\M_\sigma$.

For any $F \in \F$, let $\alpha_F$ be the number of cyclic permutations containing $F$ as a middle interval and $\beta_F$ be the number of cyclic permutations containing $F$ as an isolated interval.    Our proof considers the tradeoffs associated with these two possibilities.  We will need to know the relative frequency with which they occur.  To this end, define
\begin{equation*}
c = \max_{F \in \F_m} \frac{\alpha_F}{\beta_F}.
\end{equation*}

For a fixed cyclic permutation $\sigma$, let $m_\sigma,i_\sigma,t_\sigma$ and $b_\sigma$ denote the weight of the collections $\M_\sigma,\I_\sigma,\T_\sigma$ and $\B_\sigma$ respectively.  Define
\begin{displaymath}
R =  n \Sigma_I(n,2) =
\begin{cases}
n \binom{n}{\floor{\frac{n}{2}}+1} + n \binom{n}{\floor{\frac{n}{2}}+2}, &\mbox{if $n$ is odd}\\
 \frac{n}{2} \binom{n}{\frac{n}{2}} + n \binom{n}{\frac{n}{2}+1} + \left(\frac{n}{2}-2 \right) \binom{n}{\frac{n}{2}+2}, &\mbox{if $n$ is even.}
 \end{cases}
\end{displaymath}
Thus, our aim is to show $\abs{\F} \le R/n$.

\begin{lemma}
\label{mainlemma}
If for each cyclic permutation $\sigma$ we have $t_{\sigma} + b_{\sigma}+ (1 + c) i_{\sigma} \le R$, then $\abs{\F} \le R/n$.
\end{lemma}

\begin{proof}
It suffices to show that
\begin{equation*}
n! \abs{\F} = \sum_{\sigma} \sum_{F \in \F^{\sigma}} \binom {n}{\abs{F}} \le (n-1)! R.
\end{equation*}
For a given $\sigma$ we have
\begin{equation}
\label{keyeq}
\sum_{F \in \F^{\sigma}} \binom {n}{\abs{F}} = t_{\sigma} + b_{\sigma} + i_{\sigma} + m_{\sigma} \le R + m_{\sigma} - c i_{\sigma}.
\end{equation}
Summing both sides of \eqref{keyeq} over all cyclic permutations, we get 
\begin{equation*}
\sum_{\sigma} \sum_{F \in \F^{\sigma}} \binom {n}{\abs{F}} \le \sum_{\sigma}\left( R + m_{\sigma} - c i_{\sigma}\right) = (n-1)!R + \sum_{F \in \F_m} {(\alpha_F - c \beta_F) \binom{n}{\abs{F}} } - \sum_{F \not \in \F_m} c \beta_F \binom{n}{\abs{F}}.
\end{equation*}
Now, since for every $F \in \F$ we have $\alpha_F - c \beta_F \le 0$ (by the definition of $c$), our lemma follows.
\end{proof}


\begin{lemma}\label{lem7}
\label{cestimate}
If $\F$ is $B$-free and contains only sets of size at least $2$ and at most $n-2$, then for each $F \in \F_m$ we have
\begin{displaymath}
\frac{\beta_F}{\alpha_F} \ge \frac{\abs{F}(n-\abs{F})}{4} - \frac{n}{2} + 1.
\end{displaymath}
\end{lemma}

\begin{proof}
Assume that $A  \subset F \subset B$.  The number of cyclic permutations containing $A,F$ and $B$ is
\begin{displaymath}
\alpha_F = \abs{A}!(\abs{F}-\abs{A}+1)!(\abs{B}-\abs{F}+1)!(n-\abs{B})!.
\end{displaymath}
The number of cyclic permutations containing only $F$ is (by inclusion/exclusion)
\begin{displaymath}
\beta_F = \abs{F}!(n-\abs{F})! - \abs{A}!(\abs{F}-\abs{A}+1)!(n-\abs{F})! - \abs{F}!(\abs{B}-\abs{F}+1)!(n-\abs{B})! + \alpha_F.
\end{displaymath}
So we have 

\begin{dmath*}
\frac{\beta_F}{\alpha_F} = 1 + \frac{\abs{F}!(n-\abs{F})!}{\abs{A}!(\abs{F}-\abs{A}+1)!(\abs{B}-\abs{F}+1)!(n-\abs{B})!} - \frac{(n-\abs{F})!}{(\abs{B}-\abs{F}+1)!(n-\abs{B})!} - \frac{\abs{F}!}{\abs{A}!(\abs{F}-\abs{A}+1)!} =  \left( \frac{(n-\abs{F})!}{(\abs{B}-\abs{F}+1)!(n-\abs{B})!} - 1 \right)  \left( \frac{\abs{F}!}{\abs{A}!(\abs{F}-\abs{A}+1)!} - 1 \right) 
\ge \min_{B} \left( \frac{(n-\abs{F})!}{(\abs{B}-\abs{F}+1)!(n-\abs{B})!} - 1\right) \cdot \min_{A}  \left( \frac{\abs{F}!}{\abs{A}!(\abs{F}-\abs{A}+1)!} - 1 \right). 
\end{dmath*}
The first term is minimized by taking $\abs{B} = \abs{F}+1$, and the second term is minimized by taking $\abs{A} = \abs{F}-1$.  By substituting these values in the inequality above, we get

\begin{align*}
\frac{\beta_F}{\alpha_F} &\ge \left(\frac{n-\abs{F}}{2} - 1\right) \left(\frac{\abs{F}}{2} - 1\right) \\
&= \frac{\abs{F}(n-\abs{F})}{4} - \frac{n}{2} + 1.  \qedhere
\end{align*}
\end{proof}

\begin{note}
\label{boundingc}
If the middle sets in $\F$ all have size at least $3$ and at most $n-3$, then for each $F \in \F_m$,
\begin{displaymath}
\frac{\beta_F}{\alpha_F} \ge \frac{\abs{F}(n-\abs{F})}{4} - \frac{n}{2} + 1 \ge \frac{n - 5}{4}.
\end{displaymath}
Therefore,
\begin{displaymath}
c = \max_{F \in \F_m} \frac{\alpha_F}{\beta_F} \le  \frac{4}{n-5}.
\end{displaymath}
\end{note}

\begin{lemma}\label{lem8}
If $\F$ is $B$-free and contains a set of size $1$ or  $n-1$, then $\abs{\F} < \Sigma_I(n,2)$ for $n \ge 17$.
\end{lemma}

\begin{proof}

Assume that $\F$ contains a set of size $n-1$, say $S$.  We define two subfamilies of $\F$.  Denote by $\F_1$ the family of those sets in $\F$ which are properly contained in $S$ and set $\F_2 = \F \setminus (\F_1 \cup \{S\})$.  Since $\F$ is $B$-free, it follows that $\F_1$ has no three sets $A,B,C$ with $A,B \subset C$.  Thus, using Theorem~\ref{thm:KT} applied to an $n-1$ element ground set we have 
\begin{displaymath}
\abs{\F_1} \le \left( 1+\frac{2}{n-1} \right )\binom{n-1}{\floor{\frac{n-1}{2}}}.
\end{displaymath}
 Since every set in $\F_2$ contains a fixed element, $x$, we can delete $x$ from each element of $\F_2$ to form a new family $\F_2' = \{F\setminus \{x\}:F\in \F_2\}$.  Clearly $\F_2'$ is also $B$-free and $\abs{\F_2} = \abs{\F_2'}$.  By Theorem \ref{thm:DKS} applied to an $n-1$ element ground set, we have
 \begin{displaymath}
 \abs{\F_2} = \abs{\F_2'} \le \binom{n-1}{\floor{\frac{n-1}{2}}} + \binom{n-1}{\floor{\frac{n-1}{2}}+1}.
 \end{displaymath}
One can easily verify this implies $\abs{\F} = \abs{\F_1} + \abs{\F_2} + 1< \Sigma_I(n,2)$ for $n \ge 17$.

If $\F$ contains a set of size $1$, a similar proof works by symmetry (note that we do not use the intersecting property here).
\end{proof}

We will use the following special case of Lemma \ref{cyclelemma}, which will be proved in Section \ref{sectionksp}:

\begin{lemma}\label{lem9}
\label{2sperner_cycle}
Let $\G$ be an intersecting $2$-Sperner collection of intervals along a cyclic permutation $\sigma$, then
\begin{equation}
\label{cycleeq}
\sum_{G \in \G} \binom{n}{\abs{G}} \le n \Sigma_I(n,2).
\end{equation}
Equality holds in \eqref{cycleeq} if and only if:
\begin{itemize}
\item $n$ is odd and $\G = \h_{0,n,2}^\sigma$;
\item $n$ is even and $\G=\h_{x,n,2}^\sigma$ for some $x \in [n]$. 
\end{itemize}
\end{lemma}

Now we are ready to prove our main theorem.

\begin{proof}[Proof of Theorem \ref{intBfree}]
Let $\sigma$ be a cyclic permutation.    By Lemma \ref{mainlemma}, it is enough to prove 
\begin{equation}
\label{mainineq1}
t_{\sigma} + b_{\sigma}+ (1 + c) i_{\sigma}~\le R. 
\end{equation}

Note that as $m_{\sigma}$ does not appear in the inequality, we can delete the middle intervals and consider $\F^\sigma_0=\F^\sigma\setminus \M_\sigma$ instead of $\F^\sigma$. The partition to top, bottom and isolated intervals remains the same. As $\F^\sigma_0$ is $2$-Sperner, we have  $t_{\sigma} + b_{\sigma}+ i_{\sigma}~\le R$ by Lemma \ref{2sperner_cycle}. Let $I$ be the number of isolated intervals in $\F^\sigma_0$ (i.e., $I =i_{\sigma}$). If $I = 0$, then we are done.    

Assume that $n$ is even and $I > 0$.   If $I > \frac{n}{2}$, then by Lemma~\ref{isol} and Lemma \ref{isol_improvement}, the total number of intervals along $\sigma$ is less than $\frac{3n}{2}-1$. Since isolated sets form an antichain, $I \le n$ by Lemma~\ref{cycleantichain}. Thus, we get an upper bound for $t_{\sigma} + b_{\sigma}+ (1 + c) i_{\sigma}$ if we take $\frac{3n}{2}-2$ intervals of the largest possible weight. So we take as many isolated sets as possible, namely $n$ sets, since their weight is further multiplied by $1+c$. As $\F^\sigma_0$ is intersecting, there are at most $\frac{n}{2}$ intervals of size $\frac{n}{2}$ (i.e. of maximum weight), and we can take the remaining $n-2$ intervals of size $\frac{n}{2}+1$.

So the maximum value of $t_{\sigma} + b_{\sigma}+ (1 + c) i_{\sigma}$ is at most $\left(\frac{n}{2} \binom{n}{\frac{n}{2}} + \frac{n}{2} \binom{n}{\frac{n}{2}+1}\right)\left(1+c\right) + \left(\frac{n}{2}-2 \right) \binom{n}{\frac{n}{2}+1}$ (note that this family cannot actually occur: if there are $n$ isolated sets, then by Lemma \ref{isol} there cannot be any more intervals in $\F^\sigma_0$). So it is enough to show
\begin{displaymath}
\left(\frac{n}{2} \binom{n}{\frac{n}{2}} + \frac{n}{2} \binom{n}{\frac{n}{2}+1}\right)\left(1+c\right) + \left(\frac{n}{2}-2 \right) \binom{n}{\frac{n}{2}+1} < R.
\end{displaymath}
Simplifying, we get
\begin{displaymath}
\frac{n}{2}\left(\binom{n}{\frac{n}{2}} + \binom{n}{\frac{n}{2}+1}\right) c <  2\binom{n}{\frac{n}{2}+1}+\left(\frac{n}{2}-2\right)\binom{n}{\frac{n}{2}+2}.
\end{displaymath}
This is equivalent to
\begin{displaymath}
\frac{n}{2}\left(\frac{n+2}{n}\binom{n}{\frac{n}{2}+1} + \binom{n}{\frac{n}{2}+1}\right) c <  2\binom{n}{\frac{n}{2}+1}+\left(\frac{n}{2}-2\right)\frac{n-2}{n+4}\binom{n}{\frac{n}{2}+1}.
\end{displaymath}
Now  dividing through by $\binom{n} {\frac{n}{2}+1}$ and rearranging, we get
\begin{displaymath}
c <  \frac{n^2-2n+24}{2n^2+10n+8}.
\end{displaymath}

It is easy to see that if $c \le \frac{4}{n-5}$, then the above inequality holds when $n\ge 17$. On the other hand, by Lemma \ref{lem8} we may assume the conditions of Lemma \ref{lem7} are satisfied, so we can use Note \ref{boundingc} to show $c \le \frac{4}{n-5}$, as required.

Now, consider the case when there are $2 \le I \le \frac{n}{2}$ isolated intervals along $\sigma$.  By Lemma \ref{isol_improvement} it follows that the total number of intervals in $\F^\sigma_0$ is at most $2n-I-1$.
Pairing off intervals with their complements and considering the maximum weight we can obtain with $2n-I-1$ intervals, it is enough to show

\begin{displaymath}
 (1 + c) I \binom{n}{\frac{n}{2}} + \left(\frac{n}{2}-I \right) \binom{n}{\frac{n}{2}} + n \binom{n}{\frac{n}{2} + 1} + \left(\frac{n}{2} -I  -1 \right) \binom{n}{\frac{n}{2} + 2} \le R.
 \end{displaymath}
Simplifying,
\begin{displaymath}
c I \binom{n}{\frac{n}{2}}  \le (I-1) \binom{n}{\frac{n}{2} + 2}. 
\end{displaymath}
Dividing through by $\binom{n} {\frac{n}{2}}$, we get
\begin{displaymath}
 I
\ge
\frac{ \frac{n(n-2)}{(n+2)(n+4)}} {\frac{n(n-2)}{(n+2)(n+4)} - c}.
\end{displaymath}

By Note \ref{boundingc}, we have $c \le \frac{4}{n-5}$.   Substituting this bound for $c$ in the above inequality, we get that the right-hand side is strictly less than $2$ when $n \ge 18$ (note that our assumption $n \ge 17$ is still sufficient since we are in the case when $n$ is even).  

So we may assume that $I = 1$ ($n$ is even) and that the total number of intervals along  $\sigma$ is exactly $2n-2$ (if we have less than $2n-2$ intervals and $I = 1$, it can be checked easily that $t_{\sigma} + b_{\sigma}+ (1 + c) i_{\sigma}< R$ for $n \ge 17$). Let us denote by $\U$ the subfamily of maximal intervals (i.e., those intervals that are not contained in any other interval) in $\F^\sigma_0$ and the subfamily of minimal intervals (i.e., those intervals that do not contain any other interval) by $\D$. Now, if either $\U$ or $\D$ contains $n$ intervals, then, since it is an antichain, by Lemma \ref{cycleantichain} it has to be a complete level $\mathcal{L}^{\sigma}_r$, for some $r$. As the family is intersecting, we have $r> \frac{n}{2}$.


If $\D=\mathcal{L}^{\sigma}_r$, then $\U\setminus \D$ consists of $n-2$ intervals of size at least $\frac{n}{2} + 2$, and a simple calculation shows $t_{\sigma} + b_{\sigma}+ (1 + c) i_{\sigma} < R$ for $n \ge 17$. If $\U=\mathcal{L}^{\sigma}_r$ and $r \ge \frac{n}{2} + 2$, a similar calculation shows again that $t_{\sigma} + b_{\sigma}+ (1 + c) i_{\sigma} < R$ for $n \ge 17$. If $\U=\mathcal{L}^{\sigma}_r$ and $r=\frac{n}{2}+1$, then $\D$ is an intersecting antichain of intervals of size at most $\frac{n}{2}$. Now by Lemma \ref{cycleantichain_atmostnhalf}, we have $|\D|\le \frac{n}{2}$ in this case, contradicting the fact that the total number of intervals is $2n-2$.

So we can assume that both  $\U$ and $\D$ contain at most $n-1$ intervals. Since the interval in $\I_\sigma$ is both maximal and minimal, we have $\abs{\U \cap \D} \ge 1$. But then, the total number of intervals in our 2-Sperner family is $\abs{\U \cup \D} = \abs \U + \abs \D - \abs{\U \cap \D} \le 2n - 3$, a contradiction.

Now let us assume that $n$ is odd and $I > 0$. Similar to the case when $n$ is even, by Lemma \ref{isol_improvement} it follows that the total number of intervals in $\F^\sigma_0$ is at most $2n-I-1$. Since isolated sets form an antichain, $I \le n$ by Lemma \ref{cycleantichain}. Pairing off intervals with their complements and considering the maximum weight we can obtain with $2n-I-1$ intervals, it is enough to show
\begin{displaymath}
(1 + c) I \binom{n}{\frac{n+1}{2}} + \left(n-I \right) \binom{n}{\frac{n+1}{2}} + \left(n -I  -1 \right) \binom{n}{\frac{n+3}{2}} \le R = n \binom{n}{\frac{n+1}{2}} + n \binom{n}{\frac{n+3}{2}}.
\end{displaymath}
Simplifying, we get 
\begin{displaymath}
c I \binom{n}{\frac{n+1}{2}} \le \left (I+1 \right) \binom{n}{\frac{n+3}{2}}.
\end{displaymath}
Dividing through by $\binom{n}{\frac{n+1}{2}}$ we get,
\begin{displaymath}
c I \le \frac{n-1}{n+3} \left(I+1 \right).
\end{displaymath}
Since $c \le \frac{4}{n-5}$ the above inequality holds if $\frac{4}{n-5} \le \frac{n-1}{n+3}$ and it can be easily checked that this is true for $n \ge 11$.

We now establish the cases of equality. First let us notice that by Lemma \ref{mainlemma}, we have $\abs{\F} = \frac{R}{n}$ if and only if we have equality in \eqref{mainineq1} for each $\sigma$. However, we just saw that if $I > 0$, the inequality~\eqref{mainineq1} is never sharp when $n$ is large enough (for both the $n$ even case and $n$ odd case). Thus, we have $I = 0$ for every $\sigma$, and by Lemma \ref{lem9} $\F_0^\sigma=\h_{0,n,2}^\sigma$ or $\F_0^\sigma=\h_{x,n,2}^\sigma$ (depending on the parity of $n$). This implies that $\F^\sigma$ is 2-Sperner for every $\sigma$. If there is a chain of length three in $\F$, then clearly there is a $\sigma$ such that $\F^\sigma$ contains all three members of this chain, a contradiction. Therefore, $\F$ is $2$-Sperner and so the equality cases follow from Theorem \ref{kspint}.
\end{proof}

\section{Bollob\'{a}s and Greene-Katona-Kleitman-type inequalities}
\label{ButterflyBollobas}
In this section we will prove Theorem \ref{bollobas2}.  

\begin{proof}
Following Bollob\'as's proof \cite{bollobas1973sperner} of Theorem \ref{bollobas}, we will use the weight function
\begin{displaymath}
w(F,\sigma) = 
\begin{cases}
\frac{1}{\abs{F}},  &\mbox{if $F \in \F$ and $F$ is an interval in $\sigma$} \\
0, &\mbox{otherwise.}
\end{cases}
\end{displaymath}
On the one hand, we have
\begin{displaymath}
\sum_{F\in \F} \sum_\sigma w(F,\sigma) = \sum_{F\in \F} (\abs{F}-1)!(n-\abs{F})!.
\end{displaymath}
We will show 
\begin{displaymath}
\sum_\sigma \sum_{F\in \F} w(F,\sigma) \le 2(n-1)!.
\end{displaymath}
Fix a cyclic permutation $\sigma$.  As before, let $\I_\sigma$ be the collection of isolated intervals along $\sigma$.  Similarly, let $\M_\sigma$ be the collection of middle intervals along  $\sigma$. For a family $\A$ let the weight of the family be $w(\A)=\sum_{A\in \A}w(A)$. Then, we claim that the following inequality holds:
\begin{equation}
\label{desiredineq}
w(\F^\sigma) \le 2  + w(\M_\sigma) - w(\I_\sigma).
\end{equation}
Indeed, initially leave out all intervals in $\M_\sigma$ and $\I_\sigma$.  The remaining intervals may be partitioned into two antichains along $\sigma$, say $\A_1$ and $\A_2$.  Clearly $\A_1 \cup \I_\sigma$ is an antichain, as is $\A_2 \cup \I_\sigma$.   By the argument from \cite{bollobas1973sperner} we have $w(\A_1 \cup \I_\sigma) \le 1$ and $w(\A_2 \cup \I_\sigma) \le 1$.  Thus, summing we have
\begin{equation}
w(\A_1) + w(\A_2) + 2w(\I_\sigma) \le 2. 
\end{equation}
Rearranging and adding $w(\M_\sigma)$ to both sides yields \eqref{desiredineq}.  

Since the only possible middle intervals along a cyclic permutation are middle sets in $\F$ (that is, $\M_\sigma \subseteq \F_m^\sigma$), summing up \eqref{desiredineq} over all cyclic permutations $\sigma$, we get
\begin{equation}
\sum_\sigma \sum_{F\in \F} w(F,\sigma) \le 2(n-1)! + \sum_{F \in \F_m} \left(\frac{\alpha_F}{\abs{F}} - \frac{\beta_F}{\abs{F}} \right) - \sum_{F \not \in \F_m}\frac{\beta_F}{\abs{F}}.
\end{equation}
We have seen already by Lemma \ref{cestimate} that $\beta_F \ge \alpha_F$ for $F \in \F_m$, and the proof is complete.
\end{proof}

The proof of Theorem \ref{gkk2} uses the exact same idea but with the following weight function.
\begin{displaymath}
w(F,\sigma) =
\begin{cases}
\frac{n-\abs{F}+1}{\abs{F}}, &\mbox{if $F \in \F$, $\abs{F}\le \frac{n}{2}$ and $F$ is an interval in $\sigma$} \\
1, &\mbox{if $F \in \F$, $\abs{F}> \frac{n}{2}$ and $F$ is an interval in $\sigma$} \\
0, &\mbox{otherwise.}
\end{cases}
\end{displaymath}

This weight function was defined in \cite{greene1976extensions}, and it was shown there that the weight of an intersecting antichain is at most $n$. In particular, we obtain that $w(\A_2 \cup \I_\sigma) \le n$ and $w(\A_1 \cup \I_\sigma) \le n$.   Using this, the proof goes the same way as the proof of Theorem \ref{bollobas2}. 

\section{Results for general posets $P$}
\label{generalP}
%
%

In this section we prove Theorem \ref{oddcase}. Before we start the proof, we define the notion of a \emph{double~chain} introduced in \cite{burcsi2013method}. 

\begin{definition} [Double chain]
	Let $\emptyset=A_0 \subset A_1 \subset A_2 \subset \ldots \subset A_n = [n]$ be a maximal chain (so $\abs {A_i }= i$). The double chain associated to this chain is given by  
	\begin{displaymath}
		\D = \{ A_0, A_1, \ldots, A_n, M_1, M_2, \ldots, M_{n-1} \},
	\end{displaymath}
	where $ M_i = A_{i-1} \cup \{A_{i+1} \setminus A_i\}$. 
\end{definition}

We will now introduce the notion of a \emph{double chain-complement pair} which is the key ingredient of the proof.
\begin{definition} [Double chain-complement pair]
	Let $\D$ be a double chain. By taking the complements of the sets in $\D$ we get another double chain $\D'$. We refer to $\h = \D \cup \D'$ as a double chain-complement pair.
\end{definition}

In the rest of this section we shall work with the double chain-complement pair $\h_0 = \D_0 \cup \D'_0$ where $\D_0$ is defined by taking $A_i = [i]$; other double chain-complement pairs are related to it by permutation. Let $\pi$ be a permutation on $[n]$ and $F \subseteq [n]$, then $F^\pi$ denotes the set $\{\pi(a):a\in F\}$.  We define the double chain-complement pair $\h_0^\pi$ to be the collection $\{F^\pi: F \in \h_0\}$. We will use a weighted double counting argument on the pairs $(F,\pi)$ where $F \in \F$, $\pi$ is a permutation, and $F\in \h_0^\pi$.  Note that for two different permutations $\pi_1$ and $\pi_2$, the pairs $(F,\pi_1)$ and $(F,\pi_2)$ are considered distinct even if $\h_0^{\pi_1} = \h_0^{\pi_2}$.



Define a weight function $w(F,\pi)$ by
\begin{equation*}
	w(F,\pi) = \begin{cases}
		\binom{n}{\abs{F}}, &\text{if $F\in \F$, $F \not = [n]$ and $F \in \h_0^\pi$}\\
		4, &\text{if $F \in \F$, $F = [n]$ and $F \in \h_0^\pi$}\\
		0, &\text{otherwise.}
	\end{cases}
\end{equation*}

We want to compute $\sum_F \sum_{\pi} w(F,\pi) $ in two different ways. First let us fix a $F \in \F$ and determine for how many permutations $\pi$ we have $F \in \h_0^\pi$. If $F = [n]$ we know that for all $n!$ permutations, $F \in \h_0^\pi$. So let us assume $F \not = [n]$. Let $H_1, H_2, H_3, H_4$ be the four sets in $\h_0$ of size $\abs{F}$ (our assumption $n \ge 4$ ensures there are four distinct sets of this size).  The number of permutations $\pi$ such that a given $H_i$ (where $1 \le i \le 4$) is mapped to $F$ is $\abs{F}!(n-\abs{F})!$, since we can map the elements of $H_i$ to $F$ arbitrarily and the elements of $[n] \setminus H_i$ to $[n] \setminus F$ arbitrarily. So it follows that the number of permutations $\pi$ such that $F \in \h_0^\pi$ is $4\abs{F}!(n-\abs{F})!$. Thus, we have

\begin{equation}
	\label{eq:oneside}
	\sum_{F\in \F} \sum_{\pi} w(F,\pi) = 4 \abs{\F} n!.
\end{equation}

Now let us fix a permutation $\pi$. Our aim is to bound the total weight of the family of sets on $\h_0^\pi$ from $\F$.  We recall a lemma due to Burcsi and Nagy \cite{burcsi2013method}.

\begin{lemma} [Burcsi-Nagy \cite{burcsi2013method}]
	Let $P$ be a  poset. Any subset of size $\abs{P} + h(P) - 1$ of a double chain contains $P$ as a subposet.
\end{lemma}

Since a $P$-free family has at most $\abs{P} + h(P) - 2$  sets on a double chain, it follows that we can have at most $2(\abs{P} + h(P) - 2)$ sets in $\F \cap \h_0^\pi$.  Since $n$ is odd, there are $8$ sets in $\h_0^\pi$ of the largest weight, $\binom{n}{\floor{\frac{n}{2}}+1}$, and $8$ sets of second largest weight, $\binom{n}{\floor{\frac{n}{2}}+2}$, and so on.  The $8$ sets of $\h_0^\pi$ of the same weight $\binom{n}{\floor{\frac{n}{2}}+i}$ (where $i \ge 1$) consist of $4$ sets and their respective complements.  Thus, at most $4$ of them can belong to our family $\F$ (since $\F$ is intersecting). It is easily seen that the largest weight we can obtain comes from taking 4 sets of the largest weight in $\h_0^\pi$, 4 sets of the second largest weight in $\h_0^\pi$, and so on. It follows that the total weight of sets in $\F \cap \h_0^{\pi}$ is at most $ \sum_{i = 1}^{\frac{2(\abs{P} + h(P) - 2)}{4}} 4 \binom{n}{\floor{\frac{n}{2}}+i}$.  So we have

\begin{equation}
	\label{eq:otherside}
	\sum_{\pi} \sum_{F\in\F}  w(F, \pi) \le n! \left ( \sum_{i = 1}^{\frac{\abs{P} + h(P)}{2}-1} 4 \binom{n}{\floor{\frac{n}{2}}+i} \right).
\end{equation}
Combining \eqref{eq:oneside} and \eqref{eq:otherside}, we have the desired bound.

\section{Intersecting $k$-Sperner families}
\label{sectionksp}

%
%
%

The aim of this section is to prove Theorem \ref{kspint}. 

\begin{lemma}
	\label{cyclelemma}
	Let $\G$ be an intersecting $k$-Sperner collection of intervals along a cyclic permutation $\sigma$, then
	\begin{equation}
		\label{cyclelemmaeq}
		w(\G)=\sum_{G \in \G} \binom{n}{\abs{G}} \le n \Sigma_I(n,k).
	\end{equation}
	Assume $k < \frac{n}{2}$, then equality holds in \eqref{cyclelemmaeq} if and only if:
	\begin{itemize}
		\item $n$ is odd and $\G = \h_{0,n,k}^\sigma$;
		\item $n$ is even, $k=1$ and $\G = \h_{0,n,1}^\sigma$ or $\G = \h_{x,n,1}^\sigma$ for some $x \in [n]$;
		\item $n$ is even, $k>1$ and $\G=\h_{x,n,k}^\sigma$ for some $x \in [n]$. 
	\end{itemize}
\end{lemma}

\begin{proof}
	First, fix $k$ and suppose that $n$ is odd.  	Following an argument of Mirsky \cite{mirsky}, $\G$ can be decomposed into $k$ antichains in the following way.  For $1 \le i \le k$ set
	\begin{displaymath}
		\G_{i} = \{G:G\in \G \mbox{ and the longest chain in $\G$ with maximal element $G$ has length $i$} \}.
	\end{displaymath}
	By Lemma \ref{cycleantichain} each $\G_i$ can have size at most $n$ and so we have $\abs{\G} \le kn$.  For any interval $G$ along $\sigma$, it is easy to see that $[n]\setminus G$ is also an interval along $\sigma$ and has size $n-\abs{G}$.  Since our family is intersecting, by pairing off each $G$ with $[n]\setminus G$, we see that $\G$ contains at most $n$ intervals of size $\floor{\frac{n}{2}}$ or $\floor{\frac{n}{2}}+1$ and at most $n$ intervals of size $\floor{\frac{n}{2}}-1$ or $\floor{\frac{n}{2}}+2$ and so on.  Thus, the bound
	\begin{displaymath}
		\sum_{G \in \G} \binom{n}{\abs{G}}  \le n \Bigg(\binom{n}{\floor{\frac{n}{2}}+1} + \binom{n}{\floor{\frac{n}{2}}+2} + \dots + \binom{n}{\floor{\frac{n}{2}}+k}\Bigg) = n \Sigma_I(n,k)
	\end{displaymath}
	is immediate. Assume now that $\G$ attains this weight and $k \le \frac{n}{2}$, then $\G$ must contain $n$ sets from each of $\LL_{\floor{\frac{n}{2}}}^\sigma \cup \LL_{\floor{\frac{n}{2}}+1}^\sigma$,  $\LL_{\floor{\frac{n}{2}}-1}^\sigma \cup \LL_{\floor{\frac{n}{2}}+2}^\sigma$, \dots, $\LL_{\floor{\frac{n}{2}}-k+1}^\sigma \cup \LL_{\floor{\frac{n}{2}}+k}^\sigma$.  In particular, we must have $\abs{\G} = kn$.  
	
	Observe that each $\G_i$ is an antichain and, since $\abs{\G} = kn$, we have $\abs{\G_i} = n$ for all $i$.  Then, Lemma~\ref{cycleantichain} implies that each $\G_i$ is equal to a level  $\LL_j^\sigma$ for some $j$.  It follows that $\G_i$ must consist of intervals of size at least $\floor{\frac{n}{2}}+1$.  Thus, assuming $\G$ is of maximal weight, we have
	\begin{displaymath}
		\G_i = \LL_{\floor{\frac{n}{2}}+i}^\sigma
	\end{displaymath}
	for each $i$ and so 
	\begin{displaymath}
		\G = \LL_{\floor{\frac{n}{2}}+1}^\sigma \cup \LL_{\floor{\frac{n}{2}}+2}^\sigma \cup \dots \cup \LL_{\floor{\frac{n}{2}}+k}^\sigma = \h_{0,n,k}^\sigma.
	\end{displaymath}
	
	Next, we consider the case when $n$ is even and $k=1$. By Lemma \ref{cycleantichain}, if $\abs{\G} = n$, then $\G$ is equal to a level $\LL_i^\sigma$ for some $i$. By the intersection property we have $i \ge \frac{n}{2}+1$ and so the weight of the family is bounded by $n \binom{n}{\frac{n}{2}+1}$ with equality only if $\G = \LL_{\frac{n}{2}+1}^\sigma$.  If $\abs{\G} \le n-1$ then, since we can take at most $\frac{n}{2}$ intervals of size $\frac{n}{2}$, the weight is bounded by $\frac{n}{2} \binom{n}{\frac{n}{2}} + (\frac{n}{2}-1)\binom{n}{\frac{n}{2}+1}$. This bound can only be attained if $\abs{\G}=n-1$, and it follows by Lemma \ref{2levels} that $G$ is pair-contiguous which, in the case $k=1$, implies $\G = \h_{x,n,1}^\sigma$ for some $x \in [n]$.  Since $\frac{n}{2} \binom{n}{\frac{n}{2}} + (\frac{n}{2}-1)\binom{n}{\frac{n}{2}+1} = n \binom{n}{\frac{n}{2}+1}$, both the $\abs{\G} = n-1$ case and the $\abs{\G} = n$ case yield optimal configurations.
	
		Finally, we consider the case when $n$ is even and $k>1$. Suppose first that for all $i,j$, $\G_i\neq \LL_j^\sigma$ (i.e. none of $\G_1, \dots, \G_k$ are equal to levels).  Then, by Lemma \ref{cycleantichain} we have $\abs{\G_i} \le n-1$ for all $i$.  We have $\abs{\G} \le kn-k$, and we will see that if $\G$ has maximal weight, then in fact $\abs{\G} \ge kn-k$.   Indeed, by pairing off intervals with their complements, we can have at most $\frac{n}{2}$ intervals of size $\frac{n}{2}$ in $\G$, at most $n$ intervals of size $\frac{n}{2}-1$ or $\frac{n}{2}+1$ in $\G$ and so on.  Thus, the total weight we can achieve with $kn-k$ intervals is bounded by 
	\begin{displaymath}
		\frac{n}{2}\binom{n}{\frac{n}{2}} + n \binom{n}{\frac{n}{2}+1} + \dots + n \binom{n}{\frac{n}{2} + k - 1} + \left (\frac{n}{2}-k \right ) \binom{n}{\frac{n}{2}+k} = n \Sigma_I(n,k),
	\end{displaymath}
	and if we have fewer than $kn-k$ intervals then the weight will be strictly less than this. This proves the upper bound, so let us assume that $k<\frac{n}{2}$ and consider the extremal families. It follows that we may assume $\abs{\G} = kn-k$ and $\abs{\G_i} = n-1$ for all $1\le i \le k$.  By Lemma \ref{2levels} each $\G_i$ is pair-contiguous on $\LL^\sigma_j\cup \LL^\sigma_{j+1}$ for some $j$.  If $j<\frac{n}{2}$, then the corresponding $\G_i$ would have size at most $\frac{n}{2}$  by Lemma \ref{cycleantichain_atmostnhalf} (contradicting $\abs{\G_i} = n-1$).  Thus, we may assume that $j \ge \frac{n}{2}$. Moreover, we have seen that the only way to attain the maximal weight ($n \Sigma_I(n,k)$) is by taking $\frac{n}{2}$ intervals of weight $\binom{n}{\frac{n}{2}}$, $n$ intervals of weight $\binom{n}{\frac{n}{2}+t}$ for $1 \le t \le k-1$, and $\frac{n}{2}- k$ intervals of weight $\binom{n}{\frac{n}{2}+k}$.  It follows that $\cup_{i=1}^k \G_i \subset \cup_{i=1}^{k+1} \LL_{\frac{n}{2}+i-1}^\sigma$ where $\G_i$ is a pair-contiguous subset of $\LL_{\frac{n}{2}+i-1}^\sigma \cup \LL_{\frac{n}{2}+i}^\sigma$ and $\G_1$ contains $\frac{n}{2}$ intervals of size $\frac{n}{2}$ and $\frac{n}{2}-1$ intervals of size $\frac{n}{2}+1$.  In turn, the family $\G_2$ is forced to be the (unique) pair-contiguous family consisting of the remaining $\frac{n}{2}+1$ intervals of size $\frac{n}{2}+1$ and $\frac{n}{2}-2$ intervals of size $\frac{n}{2}+2$.  The structures of $\G_3$ through $\G_k$ are forced in a similar way, and we will obtain the required equality $\G = \h_{x,n,k}^\sigma$ for some $x \in [n]$.  Indeed, $\G_1$ being pair-contiguous with $\frac{n}{2}$ intervals of size $\frac{n}{2}$ has the property that there is an element, $x$, such that $x$ belongs to every lower interval of $\G_1$ and none of the upper intervals of $\G_1$.  Similarly, $x$ belongs to every lower interval of $\G_2$ and none of the upper intervals of $\G_2$. Continuing in this way we obtain $\G_k$ where the lower intervals contain $x$ and the upper intervals do not, and it follows that the union of the $\G_i$'s, that is $\G$, is precisely equal to $\h_{x,n,k}^\sigma$.


The remaining case is that $\G_i=\LL_j^\sigma$ for some $i,j$. We will show that $\G$ cannot have maximal weight and	this will complete the proof since we have already classified the extremal families in the case that no $\G_i$ is equal to a level $\LL_j^\sigma$.   Suppose, by way of contradiction,  that $s$ is the smallest integer such that $\G_s$ is equal to a level, say $\LL_t^\sigma$ ($t>\frac{n}{2}$).   The weight of $\G_s \cup \G_{s+1} \cup \dots \cup \G_k$ is clearly bounded by 
	\begin{displaymath}
		n \binom{n}{t} + n \binom{n}{t+1} + \dots +n \binom{n}{t + k - s}.
	\end{displaymath}
	If $t > \frac{n}{2} + s-1$, then, by the previous case (no full levels), the weight of $\G_1 \cup \dots \G_{s-1}$  is maximized by taking
	\begin{displaymath}
		\G_1 \cup \dots \cup \G_{s-1} = \h_{x,n,s-1}^\sigma,
	\end{displaymath}
	for some $x \in [n]$.  The weight of $\G_1 \cup \dots \cup \G_{s-1}$ is 
	\begin{displaymath}
		w(\G_1 \cup \dots \cup \G_{s-1}) = \frac{n}{2} \binom{n}{\frac{n}{2}} + n \binom{n}{\frac{n}{2}+1} + \dots + n \binom{n}{\frac{n}{2}+s-2} + \left (\frac{n}{2} - (s-1) \right) \binom{n}{\frac{n}{2}+s-1},
	\end{displaymath}
	and it follows that the total weight of $\G$ is at most
	\begin{multline}
		\frac{n}{2} \binom{n}{\frac{n}{2}} + n \binom{n}{\frac{n}{2}+1} + \dots + n \binom{n}{\frac{n}{2}+s-2} \\ +\left(\frac{n}{2} - (s-1) \right) \binom{n}{\frac{n}{2}+s-1} + n \binom{n}{\frac{n}{2}+s} + n \binom{n}{\frac{n}{2}+s+1} + \dots +n \binom{n}{\frac{n}{2}+ k}.
	\end{multline}
	Subtracting $w(\G)$ from $w(\h_{x,n,k}^\sigma)$ we obtain 
	\begin{multline}
		w(\h_{x,n,k}^\sigma)-w(\G) \ge \left(\frac{n}{2}+s - 1 \right) \binom{n}{\frac{n}{2}+s-1} - \left(\frac{n}{2}+ k\right) \binom{n}{\frac{n}{2}+k}\\  = n \Bigg(\binom{n-1}{\frac{n}{2}+s-2}- \binom{n-1}{\frac{n}{2} + k -1}\Bigg) > 0,
	\end{multline}
	which implies $w(\G)$ is not maximal.
	
	Next, consider the case when $t \le \frac{n}{2} +s -1$.  By pairing off intervals with their complements along $\sigma$, it follows that 
	\begin{displaymath}
		w(\G_1 \cup \dots \cup \G_{s-1}) \le \frac{n}{2} \binom{n}{\frac{n}{2}} + n \binom{n}{\frac{n}{2}+1} + \dots + n \binom{n}{ t-1}.
	\end{displaymath}
	Thus, the whole weight is 
	\begin{displaymath}
		w(\G) \le \frac{n}{2} \binom{n}{\frac{n}{2}} + n \binom{n}{\frac{n}{2}+1} + \dots + n \binom{n}{t-1} + n \binom{n}{t} + \dots +n \binom{n}{t +k - s},
	\end{displaymath}
	but $t-s \le \frac{n}{2}-1$ so
	\begin{displaymath}
		w(\G) \le \frac{n}{2} \binom{n}{\frac{n}{2}} + n \binom{n}{\frac{n}{2}+1} + \dots  +n \binom{n}{\frac{n}{2}-1 +k} < w(\h_{x,n,k}^\sigma).
	\end{displaymath}
	Thus, we may conclude that we never have $\G_i=\LL_j^\sigma$ for any $i,j$.  It follows that the only possible equality case is $\G = \h_{x,n,k}^\sigma$ for some $x \in [n]$.
\end{proof}

\begin{proof}[Proof of Theorem \ref{kspint}]

	By Lemma \ref{cyclelemma}, we have that for every $\sigma$,
	\begin{equation}
		\label{useforequality}
		\sum_{F \in \F^\sigma} w(F,\sigma) \le n \Sigma_I(n,k).
	\end{equation}
	By the double counting outlined in Section \ref{Sectioncycle}, it is immediate that $\abs{\F} \le \Sigma_I(n,k)$.  Thus, it remains to determine the possible extremal families.  If $\F$ is extremal, then for every $\sigma$ we have equality in \eqref{useforequality} and so we are in an equality case given by Lemma \ref{cyclelemma}. Let us consider the case $k<n/2$.

	Assume first that $n$ is odd, then for every $\sigma$ we have that $\F^\sigma$ is equal to $\h_{0,n,k}^\sigma$.  In this case, it is immediate that $\F = \h_{0,n,k}$.

	Suppose now that $n$ is even and $k=1$.  There are two cases: either there exists a $\sigma$ for which $\F^\sigma = \h_{0,n,1}^\sigma$ or there does not.  Assume that we have $\F^\sigma =\h_{0,n,1}^\sigma$, and form a new cyclic permutation $\sigma'$ by transposing two adjacent elements of $\sigma$.  Observe that $\F^{\sigma'}$ still contains $n-2$ of the same intervals on level $\frac{n}{2}+1$ (namely, those without exactly one of the transposed elements). Now, configurations of the form $\h_{x,n,1}^\sigma$, $x \in [n]$, have $\frac{n}{2} -1$ intervals of size $\frac{n}{2}+1$. Thus, we have that $\F^{\sigma'}$ must have the form $\h_{0,n,1}^{\sigma'}$.  Since every permutation can be generated by transpositions of consecutive elements, it follows that for all $\sigma$, $\F^\sigma = \h_{0,n,1}^\sigma$ and so $\F = \h_{0,n,1}$.  Thus, we will assume that for all $\sigma$ we have $\F^\sigma =  \h_{x,n,1}^\sigma$, $x \in [n]$.
	
	If $n$ is even and $k >1$ and $\F^\sigma =\h_{0,n,k}^\sigma$ for some $\sigma$, then in a completely analogous way to the above $k=1$ case we can deduce that $\F = \h_{0,n,k}$. However, for $k>1$ we have $\abs{\h_{0,n,k}} < \abs{\h_{x,n,k}}$.  Indeed, simply observe
	
	\begin{align*}
		\abs{\h_{x,n,k}} - \abs{\h_{0,n,k}} &= \binom{n-1}{\frac{n}{2}-1}+ \binom{n-1}{\frac{n}{2}+k} - \binom{n}{\frac{n}{2}+k} \\
		&= \binom{n-1}{\frac{n}{2}-1} - \binom{n-1}{\frac{n}{2}+k-1} \\
		&> 0.
	\end{align*}
	Thus, we may rule out the $\h_{0,n,k}$ case for $k > 1$ and conclude that $\F^\sigma \not =\h_{0,n,k}^\sigma$ for any $\sigma$.

	So, finally, we may suppose that $n$ is even and $k \ge 1$ and that for every $\sigma$, we have $\F^\sigma = \h_{x,n,k}^\sigma$ for some $x \in [n]$. We want to show that $\F = \h_{x,n,k}$ for some $x$.  Each cyclic permutation contains $\frac{n}{2}$ intervals of size $\frac{n}{2}$ and $n$ intervals of size $\frac{n}{2}+i$ for $1 \le i \le k-1$ and $\frac{n}{2} -k$ intervals of size $\frac{n}{2}+k$. Therefore all the sets of $[n] \choose \frac{n}{2}+i$ for $1 \le i \le k-1$ are in $\F$. It only remains to show that $\F$ contains all the sets of size $\frac{n}{2}$ that contain a fixed element and all the sets of size $\frac{n}{2}+k$ that do not contain that fixed element. 
	
	We supposed that for each $\sigma$, $\F^{\sigma}$ contains all of the $\frac{n}{2}$-element intervals containing some $x$ and all the $(\frac{n}{2}+k)$-element intervals not containing that $x$. However, the $x$'s corresponding to different $\sigma$'s may be different. We claim that this is impossible; i.e., that $\F$ contains all of the $\frac{n}{2}$-element intervals containing some fixed element $x$ and all the $(\frac{n}{2}+k)$-element intervals not containing that $x$. First, let us fix an arbitrary cyclic permutation $\sigma$. Since we assumed that there is an element $x$ such that $\F^{\sigma}$ contains all of the $\frac{n}{2}$-element intervals containing $x$, notice that we have two $\frac{n}{2}$-element sets $A$ and $B$ in $\F$ that are intervals along this $\sigma$ such that $A \cap B = \{x\}$. We want to show that every $\frac{n}{2}$-element set in $\F$ contains this $x$. Suppose, by contradiction, that there exists an $\frac{n}{2}$-element set $C$ (in $\F$) not containing $x$. Observe that $\abs{[n] \setminus (A \cup B)} = 1$, so let $[n] \setminus (A \cup B) = \{y\}$.

	If $C$ contains $y$, then notice that we can find a cyclic permutation $\sigma'$ where $A$, $B$ and $C$ are intervals (since we can easily arrange the sets $A, B, C$ as intervals along a cycle). Since $A, B$ and $C$ do not have a common element, this is a contradiction to our assumption that all the $\frac{n}{2}$-element intervals along every cyclic permutation contain a fixed element. However, if $C$ does not contain $y$, we can find a cyclic permutation $\sigma''$ where $A$, $B$ and $C \cup \{y\}$ are intervals. Along this $\sigma''$, since we have two intervals (namely, $A$ and $B$) that intersect just in $x$, all the $\frac{n}{2}$-element intervals in $\F^{\sigma''}$ must also contain $x$ and all the $(\frac{n}{2}+k)$-element intervals in $\F^{\sigma''}$ do not contain $x$. In particular, there is an $(\frac{n}{2}+k)$-element interval $U$ which contains $C \cup \{y\}$ (this is because the interval $C \cup \{y\}$ doesn't contain $x$). Now, since all the intervals along $\sigma''$ of sizes $\frac{n}{2} +i$, $2 \le i \le k-1$ are in $\F$, it is easy to find a $(k+1)$-chain in $\F$ starting with $C$, then continuing with $C \cup \{y\}$ and ending with $U$, a contradiction to our assumption that $\F$ is $k$-Sperner. Therefore, our claim is proved. 
	
	By a standard double counting argument on pairs $(F, \sigma)$ where $F \in \F$ and $F$ is an interval along $\sigma$, we can see that $\F$ contains exactly $\binom{n-1}{\frac{n}{2}-1}$ sets of size $\frac{n}{2}$, and by the previous paragraph all the $\frac{n}{2}$-sets in $\F$ must contain a fixed element. Therefore, $\F$ contains every $\frac{n}{2}$-element set containing a fixed element and nothing else. But this means $\F$ cannot contain any set of size $\frac{n}{2}+k$ containing $x$ because otherwise we will have a $(k+1)$-chain in $\F$, a contradiction. But by the same double counting argument, we can see that $\F$ contains $\binom{n-1}{\frac{n}{2}+k}$ sets of size $\frac{n}{2}+k$, and all these sets must not contain $x$. This shows that $\F = \h_{x,n,k}$, as desired.
	
It only remains to establish the case of equality if $k=\frac{n+1}{2}$. Note that the upper bound is equal to $2^{n-1}$ in this case. Thus, an extremal family $\F$ contains either $A$ or the complement of $A$ for every $A\subset [n]$. If there is $F\in \F$ with $\abs{F} <\frac{n}{2}$, then there must be a set $F'\supset F$ with $F'\not\in \F$ by the $k$-Sperner property. Then the complement of $F'$ has to be in $\F$, but it is disjoint from $F$, a contradiction.
\end{proof}

\section*{Acknowledgements}
We would like to thank the two anonymous referees for their careful reading of the manuscript and valuable comments.  The research of all three authors was partially supported by the National Research, Development and Innovation Office NKFIH, grant K116769, and the research of the first author was also partially supported by the J\'anos Bolyai Research Scholarship of the Hungarian Academy of Sciences.

\bibliography{IntersectingP}

\end{document}